\documentclass[10pt, leqno, twoside]{article}
\usepackage[english, francais]{babel}
\usepackage[T1]{fontenc}
 \usepackage[applemac]{inputenc} 
\usepackage{amsmath}
\usepackage{amsfonts}
\usepackage{amssymb}
\usepackage{amsthm}
\usepackage{mathrsfs}
\usepackage{euscript}
\usepackage[normalem]{ulem}
\usepackage{cancel}
\usepackage{array}
\usepackage[dvipsnames]{xcolor}

\usepackage[left=2cm,right=5.5cm,top=1cm,bottom=2cm]{geometry}
\usepackage{stmaryrd}
\usepackage{tikz}
\usepackage{fancyhdr}

\usepackage{setspace}
\newtheorem{thm}{Théorème}[section]

\newtheorem{rem}[thm]{Remarque}
\newtheorem{rems}[thm]{Remarques}

\newtheorem{lem}{Lemme}[section]

\title{Entiers ultrafriables en progressions arithmétiques}
\author{Cécile Dartyge, David Feutrie \& Gérald Tenenbaum
\\
\smallskip
\centerline{\small\dateheure}\\ 
}

\usepackage{titletoc}
\numberwithin{equation}{section}
\newcommand{\md}[3]{#1 \equiv #2\hskip-2.5mm \pmod {#3} }
\newcommand{\mda}[3]{#1 \equiv #2 \hskip-2.5mm \pmod {#3} }

\newcommand{\mmd}[2]{\equiv #1\,(\mbox{mod}\,#2)}

\def\ee{{\rm e}}
\def\d{\,{\rm d}}
\def\dd{{\rm d}}
\def\0{{\boldsymbol 0}}

\def\CC{{\mathbb C}}

\def\R{{\mathbb R}}
 \def\cf{{cf.}}

\newcommand{{\gotT}}{{\mathfrak T}}
\newcommand{{\gotW}}{{\mathfrak W}}
\def\Z{{\mathbb Z}}

\newcommand{\rD}{\EuScript{D}}

\newcommand{\rV}{\EuScript{V}}
\newcommand{\rW}{\EuScript{W}}

               \countdef\temps=170
                \temps=\time
                \countdef\nminutes=171{\nminutes = \time}
                \countdef\nheure=172
                \def\heure{\begingroup
                   \temps = \time \divide\temps by 60
                   \nheure = \temps
                   \nminutes = \time
                   \multiply\temps by 60
                   \advance\nminutes by -\temps
                   \ifnum\nminutes<10 \toks1 = {0}%
                   \else\toks1 = {}%
                   \fi
                   \number\nheure h\the\toks1 \number\nminutes
                \endgroup}%

\def\dater{\vglue-10mm\rightline{(\the\day/\the\month/\the\year)}}
                \def\dateheure{(version \the\day/\the\month/\the\year,\ \heure)}
                
\providecommand{\keywords}[1]
{
  \small	
  \textbf{\bf{Keywords :}} #1
}

\begin{document}
\renewcommand{\refname}{Bibliographie}
\date{~}
\maketitle

{\leftskip10mm\rightskip10mm\noindent{\bf Abstract}\par 
A natural integer is called $y$-ultrafriable if none of the prime powers occurring in its canonical decomposition exceed $y$. We investigate the distribution of $y$-ultrafriable integers not exceeding $x$ among arithmetic progressions to the modulus~$q$. Given a  sufficiently small, positive constant $\varepsilon$, we obtain uniform estimates valid for $q\leqslant y^{c/\log_2y}$ whenever $\log y\leqslant (\log x)^\varepsilon$, and for $q\leqslant \sqrt{y}$ if $(\log x)^{2+\varepsilon}\leqslant y\leqslant x$.\par 
\medskip

\par \noindent\keywords{ultrafriable integers, friable integers, saddle-point estimates, sieve, Siegel zero, Gaussian distribution, local behaviour.}\par \medskip
\par }
\renewcommand{\contentsname}{Sommaire}

\section{Introduction}

Soit $y>0$. Un entier  positif $n$ est dit $y$-friable si son plus grand facteur premier, noté $P^+(n)$ avec la convention $P^+(1)=1$, n'excède pas $y$. Pour $x>0$, $y>0$, nous désignons par $S(x,y)$ l'ensemble des entiers $y$-friables inférieurs ou égaux à $x$ et par $\Psi(x,y)$ son cardinal. 
Les propriétés structurelles de l'ensemble $S(x,y)$ ont fait l'objet d'une abondante littérature depuis une trentaine d'années, notamment concernant sa répartition dans les progressions arithmétiques: voir par exemple Fouvry \& Tenenbaum \cite{FT}, Hildebrand \& Tenenbaum \cite{HT,HT93}, La Bretèche \& Tenenbaum \cite{RT}, Granville \cite{G1,G2}, Soundararajan \cite{Sound}, Harper \cite{Ha}. 
\par 

Quoique également susceptible d'intéressantes applications dans diverses branches des mathématiques --- \cf\ \cite{T15} et la bibliographie incluse --- une notion voisine, celle d'entier {\it ultrafriable}, a reçu beaucoup d'attention. Un entier naturel est dit $y$-ultrafriable s'il n'est divisible par aucune puissance de nombre premier excédant $y$. Pour $x>0$, $y>0$, nous désignons par $U(x,y)$ l'ensemble des entiers $y$-ultrafriables n'excédant pas $x$ et par $\Upsilon(x,y)$ son cardinal.
\par 
Des estimations  satisfaisantes de $\Upsilon(x,y)$ ont été obtenues dans \cite{T15}. Nous nous proposons ici d'aborder la question de la répartition des entiers ultrafriables dans les progressions arithmétiques en évaluant le comportement asymptotique des quantités
$$  \Upsilon_q(x,y):= \sum_{\substack{n \in U(x,y)\\ (n,q)=1 }}1  ,\qquad \Upsilon(x,y;a,q):= \sum_{\substack{n \in U(x,y)\\ \mda{n}{a}{q} }}1   $$
sous des conditions aussi peu restrictives que possibles concernant les variables $x,\,y$ et $q$. Notons toutefois qu'une prise en compte plus fine de la répartition des zéros des fonctions $L$ de Dirichlet permettrait d'étendre, dans les résultats présentés ci-dessous, le domaine de variation de la variable $q$. Nous avons préféré reporter ces complications à un prochain travail.  \par

 Désignons classiquement par $\tau(n)$ le nombre des diviseurs d'un entier naturel $n$, et par $\omega(n)$ le nombre de ses facteurs premiers, comptés sans multiplicité.  Notons par ailleurs $\pi(y)$ le nombre des nombres premiers n'excédant pas $y$.
\par 
Soit $y\geqslant 1$. Pour tout nombre premier $p$, nous avons $p^\nu\leqslant y$ si, et seulement si, $\nu\leqslant \nu_p = \nu_p(y) :=  \lfloor {\log y}/{\log p} \rfloor$. Posant
$$   \psi_q(y):=\sum_{\substack{p \leqslant y \\  p\,  \nmid \, q }}\nu_p \log p ,\qquad N_{q,y}:= \ee^{\psi_q(y)}, $$ 
il s'ensuit que $x\mapsto\Upsilon_q(x,y)$ est la fonction de comptage des diviseurs de $N_{q,y}$. Ainsi
\begin{equation*}
\Upsilon_q(x,y) = \tau(N_{q,y})=\prod_{\substack{p \leqslant y \\ p \, \nmid \, q}}(1+\nu_p)  \qquad  (x \geqslant N_{q,y})
\end{equation*}  
et donc
\begin{equation*}
\Upsilon_q(x,y)= 2^{\pi(y) +O(\sqrt{y}/\log y)} \qquad \big(x \geqslant N_{q,y}, \ \omega(q) \ll \sqrt{y}/\log y \big),
\end{equation*}
\par Comme la symétrie des diviseurs de $N_{q,y}$ autour de $\sqrt{N_{q,y}}$ implique
\begin{equation*}
\Upsilon_q(x,y)= \tau(N_{q,y})-\Upsilon_q \Big( \displaystyle \big(N_{q,y}/x\big)- ,y  \Big) \qquad \Big(\sqrt{N_{q,y}}\leqslant  x \leqslant N_{q,y}  \Big),
\end{equation*}
nous pouvons  restreindre l'étude de $\Upsilon_q(x,y)$ au cas
\begin{equation*}
x<\sqrt{N_{q,y}}, \quad \textrm{i.e.} \quad \psi_q(y)> 2 \log x.
\end{equation*}
Pour simplifier l'exposition, nous considérerons en fait le domaine légèrement étendu
\begin{equation}
\label{e11}
x\geqslant y\geqslant 2,\quad\psi(y)> 2 \log x,
\end{equation}
où $\psi(y):=\psi_1(y)$ désigne la fonction de Tchébychev. \par 
\smallskip
Nous restreignons également l'étude de $\Upsilon(x,y;a,q)$ au cas  $(a,q)=1$. Cette contrainte pourrait être levée au prix de quelques complications techniques. En effet, posant $d:=(a,q)$, de sorte que $(a/d,q/d)=1$,  nous avons $\Upsilon(x,y;a,q)=0$ si $d \not \in U(x,y)$, et, dans le cas contraire,  
\begin{equation*}
 \Upsilon(x,y;a,q)=\sum_{\substack{m\leqslant x/d\\ \mda m{a/d}{q/d}\\ p^\nu\|m\Rightarrow \nu\leqslant \nu_p-v_p(d)}}1,
\end{equation*}
où $v_p(d)$ désigne la valuation $p$-adique de $d$. Les techniques développées dans le présent travail peuvent être  adaptées pour évaluer cette quantité : il suffit essentiellement de remplacer la série de Dirichlet $Z_{q}(s,y)$ introduite en \eqref{e12} {\it infra} par $$Z_{q}(s,y)\prod_{\substack {p^\nu\|d\\ p\,\nmid\,q/d}}\frac{1-p^{-(\nu_p+1- \nu)s}}{1-p^{-s}}\cdot$$
Un exemple de l'estimation subséquente est donné à la Remarque \ref{r6} {\it infra}.\par

\section{Résultats}

\subsection{Évaluation de $\Upsilon_q(x,y)$}
Tenenbaum \cite{T15} a obtenu, pour tout $\varepsilon>0$, l'estimation
\begin{equation}\label{thm2}
\Upsilon(x,y) = \Psi(x,y) \Big \{  1 + O \Big(  \frac{u \log 2u}{\sqrt{y}\log y}  \Big)  \Big \}\qquad  \big(x \geqslant y \geqslant  (\log x)^{2+\varepsilon} \big),
\end{equation}
où, ici et dans la suite, nous notons $u:=\log x/\log y$ $(x\geqslant 2,y\geqslant 2)$.
La même méthode permet d'établir la pertinence de l'approximation de $\Upsilon_q(x,y)$ par $$\Psi_q(x,y):=\sum_{\substack{n\in S(x,y)\\(n,q)=1}}1$$ pour ces mêmes « grandes » valeurs de $y$, et un large domaine en $q$. Nous renvoyons à \cite{HT, RT,BT17} et à la bibliographie de ces travaux pour des évaluations explicites et implicites de $\Psi_q(x,y)$.
\par

\begin{thm}\label{T2}
Soit $\varepsilon  >0 $. Sous les conditions
\begin{equation}
\label{e110}
x \geqslant y \geqslant (\log x)^{2+\varepsilon},\quad P^+(q) \leqslant y, \quad     \omega(q) \ll \sqrt{y}  ,  
\end{equation}
nous avons
\begin{eqnarray}
\label{e111}
\Upsilon_q(x,y) = \Psi_q(x,y)  \bigg \{   1+O \bigg(  \frac{q u \log 2u}{\varphi(q) \sqrt{y} \log y }  \bigg)  \bigg \}  . 
\end{eqnarray}
\end{thm}

Notons que le terme d'erreur de \eqref{e111} tend bien vers 0 dans le domaine \eqref{e110}: il est en fait $\ll1/(\log x)^{\varepsilon/2}$.
\par 
L'énoncé relatif aux petites valeurs de $y$ nécessite quelques notations supplémentaires. 
Nous désignons par
\begin{equation}\label{e12}
Z_q(s,y):= \prod_{\substack{p \leqslant  y\\ p  \, \nmid \, q}}\frac{1-p^{-(\nu_p+1)s}}{1-p^{-s}} \quad ({\Re}e\ s >0).
\end{equation}
la série de Dirichlet associée à la fonction de comptage $\Upsilon_q(x,y)$, convenons de poser $Z(s,y):=Z_1(s,y)$, et notons
\begin{equation*}
\varphi_1(s,y):=- \frac{Z'(s,y)}{Z(s,y)} =\sum_{p \leqslant y}\Big \{ \frac{\log p}{p^{s}-1} - \frac{(\nu_p+1)\log p}{p^{(\nu_p+1)s}-1} \Big \} \qquad  ({\Re}e\ s >0,\ y\geqslant 2) .
\end{equation*}
Ainsi qu'observé dans  \cite{T15},  l'équation
\begin{equation}\label{e13}
\varphi_1(\sigma,y)= \log x \qquad (\sigma >0)
\end{equation}
possède, sous la condition \eqref{e11}, une unique solution $\beta=\beta(x,y)$, qui est donc le point-selle relatif à l'intégrale de Perron pour $\Upsilon(x,y)$.
\par \goodbreak
Désignons par
\begin{equation}\label{Phi}
\Phi(z):=\frac{1}{\sqrt{2\pi}} \int_{z}^{\infty} \ee^{-t^2/2}\mathrm{d}t  \qquad (z \in  \mathbb{R} )
\end{equation}
la fonction de répartition décroissante de la loi gaussienne, et posons
$$  G(z):= \ee^{z^2/2}\Phi(z)   \qquad  (z \in \mathbb{R}), $$
de sorte que
\begin{equation}\label{e14}
G(z)=\tfrac{1}{2} + O(z) \ \ \ (z \rightarrow  0 ),\quad  G(z) = \frac{1}{\sqrt{2\pi}z}  \Big \{ 1-\frac{1}{z^2}+ O \Big( \frac{1}{z^4} \Big)   \Big \} \ \ \  (z \rightarrow + \infty ).
\end{equation} \par 

L'estimation suivante a été établie dans  \cite{T15} :
\begin{equation}
\label{thm1}
\Upsilon(x,y)= x^{\beta}Z(\beta,y)G(\beta \sqrt{\sigma_2}) \Big \{   1+O \Big( \frac{1}{u} \Big)  \Big \} \qquad  \big( 2 \log x < \psi(y) \ll (\log x)^3 \big).
\end{equation}  

Nous évaluons le rapport $\Upsilon_q(x,y)/\Upsilon(x,y)$ en exploitant l'approximation \eqref{thm1}, relevant du paramètre implicite $\beta$. La démarche est analogue à celle de  \cite{RT}, où La Bretèche et Tenenbaum recourent à l'approximation de $\Psi(x,y)$ par la méthode du col --- cf. \cite{HT} --- pour évaluer $\Psi_q(x,y)/\Psi(x,y)$. Dans ce dernier cas, le col de l'intégrale de Perron est défini comme l'unique solution positive $\alpha:=\alpha(x,y)$ de l'équation
\begin{equation}
\label{defalpha}
\sum_{p \leqslant y} \frac{\log p}{p^{\alpha}-1}= \log x.
\end{equation} 
\par 
Notant $p_k$ le $k$-ième nombre premier, de sorte que $p_k\asymp k\log 2k$ $(k\geqslant 1)$, et 
\begin{equation}\label{zq}
z_q:= p_{\omega(q)} \qquad (q \geqslant 1)
\end{equation}
avec la convention  $p_0:=2$, nous posons ensuite
\begin{align} 
\label{eta}
\eta&=\eta(x,y):=\frac{ \psi(y)}{\log x} -2,
\quad\vartheta_q=\vartheta_q(y):= \frac{\log z_q}{\log y} \asymp \frac{\log \{  2+ \omega(q) \}}{\log y} \quad (q \geqslant 1),
\end{align}
et introduisons les termes d'erreur $\Delta_q=\Delta_q(x,y)$ et $D_q=D_q(x,y)$ définis par
\begin{equation} 
\label{e15}
\Delta_q:=  \begin{cases}
\displaystyle  \frac{ (\log x)^{\vartheta_q}}{\log y} \Big(  1 + \frac{1}{\vartheta_q \log (1+\eta)}    \Big)  & \text{si }  2 \log x < \psi(y) \ll (\log x)^2 , \\
\noalign{\vskip-5mm}\\
\displaystyle  \frac{\vartheta_q \{ u \log 2u  \}^{\vartheta_q}}{1+\vartheta_q \log 2u} & \text{si } y>(\log x)^2,
\end{cases}   
\end{equation}
et $D_q:=\min \{\omega(q) , \Delta_q \}$. Il est à noter que, lorsque $y= (\log x)^{1+\lambda}$ avec $\lambda \asymp 1$, les deux expressions de $\Delta_q$ apparaissant dans \eqref{e15} sont du même ordre de grandeur.\par  

Posant
\begin{equation*}
g_q(s)  :=  \prod_{p \, \mid \, q} \frac{1-p^{-s}}{1-p^{-(\nu_p+1)s}}     \qquad   (q \geqslant 1,\   s \in \mathbb{C}     ),
\end{equation*} et  notant, ici et dans la suite, $\log_k$ la $k$-ième itérée de la fonction logarithme $(k \geqslant 1)$, nous obtenons le résultat suivant, où $x_0$ désigne une constante absolue assez grande.

\begin{thm}\label{T1}
Soit $ \varepsilon>0$. Les assertions suivantes sont valides sous les conditions
\begin{align}
  \label{e16}
 &x \geqslant y \geqslant 2,\quad  x>x_0,\quad   2 \log x <\psi(y) \leqslant \exp\big( (\log x)^{1/5}/(\log_2 x)^{(1+\varepsilon)/5}  \big)   ,\\
\label{e161} 
&  q \geqslant 1, \quad  P^+(q) \leqslant y  , \quad  \omega(q) \ll  y^{1/2-(1+\varepsilon)(\log_2 u)/ \log u}.
\end{align} 
\begin{itemize}
\item[\rm(i)] Nous avons 
\begin{equation}
\label{e17}
\begin{aligned}
\Upsilon_q(x,y) &= g_q(\beta) \Upsilon(x,y)  \bigg \{  1 + O \bigg(   \frac{1+D_q^2}{u}+  \frac{D_q(1+ \eta)}{\sqrt{u}+ \eta u}   \bigg)    \bigg \} \\
&=  x^{\beta} Z_q(\beta,y)  G(\beta \sqrt{\sigma_2})  \bigg \{    1 + O \bigg ( \frac{1+D_q^2}{u}+ \frac{D_q(1+\eta)}{\sqrt{u}+\eta u}   \bigg)   \bigg \}, 
\end{aligned}
\end{equation}
et donc, lorsque $\eta \gg 1 $,
\begin{equation*}
\Upsilon_q(x,y) = g_q(\beta) \Upsilon(x,y) \Big \{   1 +O \Big ( \frac{1+\Delta_q^2}{u}   \Big)  \Big \}.
\end{equation*}
\item[\rm(ii)]
Si de plus $\eta \leqslant \tfrac12$, alors 
\begin{equation}\label{e18}
\Upsilon_q(x,y) = g_q(\beta) \Upsilon(x,y) \Big \{   1 +R +O \Big ( \frac{1+\omega(q)^2}{u}   \Big)  \Big \}
\end{equation}
avec $R \asymp  \omega(q)(1+\eta)/ \big(\sqrt{u}+\eta u \big) $.
\medskip
\item[\rm(iii)] Sous la condition supplémentaire $\eta=o\big(1/\sqrt{u}\big)$, nous avons
\begin{equation}\label{e19}
\Upsilon_q(x,y) = g_q(\beta) \Upsilon(x,y) \Big \{   1 + \frac{\omega(q)}{\sqrt{\pi u}}+ O \Big (  \eta \omega(q) + \frac{\log q}{\sqrt{u}\log y}+  \frac{\omega(q)^2}{u}   \Big)  \Big \}.
\end{equation}
\end{itemize}
\end{thm}

\begin{rems}\label{remsDqEq}
{\rm(a)} Sous la condition \eqref{e161}, et si en outre $ 0<\eta\leqslant  1$, nous avons  $\Delta_q\asymp \{1+ \omega(q)\}/\eta $ et donc $D_q\asymp\omega(q)$. 
\par 
{\rm(b)} Ainsi que l'attestent les formules \eqref{alpha} et \eqref{beta} {\rm infra}, les points-selles $\alpha$ et $\beta$ sont proches lorsque $y>(\log x)^{1+\varepsilon}$, ce qui implique alors $\Delta_q\asymp E_q$, où $E_q$ est le terme d'erreur introduit dans \cite{RT}. Cela permet de faire appel à certaines estimations établies dans \cite{RT}, notamment 
\begin{equation}
\label{majDeltaq}
\Delta_q(1+\Delta_q)\ll{\vartheta_q\log (u+1)}\ll1\quad\Big(y>(\log x)^{1+\varepsilon},\,\omega(q)\ll y^{1/\log (u+2)}\Big).
\end{equation}
\par 
Il s'ensuit en particulier que, dans les hypothèses du Théorème \ref{T1}, tous les termes d'erreur y apparaissant tendent vers 0.
 \par 
{\rm (c)} Le terme principal de la seconde formule \eqref{e17} constitue également une bonne approximation dans le domaine complémentaire 
\begin{equation}
\label{dcomp}
y>(\log x)^{2+\varepsilon},\quad\omega(q)\ll\sqrt{y}/\log y.
\end{equation}
 En effet, il résulte de \cite[Cor. 2.2]{RT}, \eqref{e111} avec $q=1$, \eqref{thm1}, et \eqref{alpha}, \eqref{beta} {\rm infra}, que, sous l'hypothèse \eqref{dcomp}, nous avons
\begin{equation*}
\label{gamq-ygrand}
\begin{aligned}
\Psi_q(x,y)&=g_q(\alpha)\Psi(x,y)\Big\{1+O\Big(\frac1u\Big)\Big\}=g_q(\alpha)\Upsilon(x,y)\Big\{1+O\Big(\frac1u\Big)\Big\}\\
&=g_q(\alpha)x^{\beta}Z(\beta,y)G(\beta \sqrt{\sigma_2}) \Big \{   1+O \Big( \frac{1}{u} \Big)  \Big \}\\
&=x^{\beta}Z_q(\beta,y)G(\beta \sqrt{\sigma_2}) \Big \{   1+O \Big( \frac{1}{u} \Big)  \Big \}.
\end{aligned}
\end{equation*}
Par \eqref{e111}, nous obtenons donc, sous les mêmes hypothèses,
\begin{equation}
\label{UP-ygrand}
\Upsilon_q(x,y) = x^{\beta}Z_q(\beta,y)G(\beta \sqrt{\sigma_2})  \bigg \{   1+O \bigg(  \frac{q u \log 2u}{\varphi(q) \sqrt{y} \log y } +\frac1u \bigg)  \bigg \}.
\end{equation} 
\par 
{\rm (d)} Comme $\Upsilon_q(x,y)$ ne dépend que du noyau sans facteur carré de $q$ et comme $\log q \leqslant \omega(q) \log y$ si $\mu(q)^2=1$ (où $\mu$ désigne la fonction de Möbius), $P^+(q) \leqslant y$, l'ordre de grandeur du second terme d'erreur de \eqref{e19} ne dépasse pas celui du terme principal complémentaire $\omega(q)/\sqrt{\pi u}$. 
\end{rems}

\subsection{Évaluation de $\Upsilon(x,y;a,q)$}
Notre approche repose, d'une part, sur une majoration des quantités
$$      \Upsilon(x,y;\chi) := \sum_{n \in U(x,y)}\chi(n),         $$
lorsque $\chi$ est un caractère de Dirichlet de module $q$ distinct du caractère principal~$\chi_0$, et, d'autre part, sur l'évaluation de Granville \cite{G2} 
\begin{equation}
\label{Gra}
\Psi(x,y;a,q):=\sum_{\substack{n\in S(x,y)\\ \md naq}}1 = \frac{\Psi_q(x,y)}{\varphi(q)}\Big \{  1 + O \Big(  \frac{\log q}{u^c \log y}+ \frac{1}{\log y}  \Big)  \Big \},
\end{equation} 
 valable, pour tout $\varepsilon>0$,  sous les conditions $x \geqslant y \geqslant q^{1+\varepsilon} $, et avec $c=c(\varepsilon)>0$.
 \par 
 Le résultat suivant rassemble nos résultats relatifs à la première voie.
Nous posons 
\begin{equation}\label{e112}
   Y_{\varepsilon}:=Y_{\varepsilon}(y)= \ee^{(\log y)^{3/2-\varepsilon}}   \qquad( \varepsilon >0,\  y \geqslant 2)
\end{equation}
 et convenons de désigner, dans toute la suite, par $c_j$ ($j=1,2, \ldots$) des constantes absolues positives convenablement choisies.

\par

\begin{thm}\label{T3}
Soit  $A>0$. Pour $\varepsilon>0$, $c_0>0$, $c_1>0$ assez petits,  sous les conditions
\begin{align}
\label{ypetit}
 &x \geqslant 2,  \quad 2 \log x < \psi(y) \leqslant \ee^{(\log x)^{\varepsilon}}, \\
\label{e113}
& 2\leqslant  q \leqslant y^{c_0/ \log _2 y},
\end{align}
et pour tout caractère de Dirichlet non principal $\chi$  de module $q$, nous avons
\begin{equation}
\label{e114}
\Upsilon(x,y;\chi) \ll  \Upsilon_q(x,y)  \Big \{   \ee^{-c_1 u/\{ 1+\vartheta(\chi)(\log u)^4 \} }  +Y_{\varepsilon}^{-1} \Big \}
\end{equation}
où $\vartheta(\chi)$ vaut 0 ou 1 et n'est non nul que si $q > (\log y)^A$ et si $\chi$ est égal à un unique caractère exceptionnel, réel.\par 
\end{thm}\par 

Le terme principal attendu pour $\Upsilon(x,y;a,q)$ est $\Upsilon_q(x,y)/\varphi(q)$, où $\varphi$ désigne la fonction indicatrice d'Euler. Nous déduisons du Théorème \ref{T3} l'évaluation suivante, valable pour les petites valeurs de $y$.   \par 

\begin{thm}\label{T4}
Soit $\varepsilon >0$. Pour $a\in\Z$, $q\geqslant 1$, $(a,q)=1$, et sous les hypothèses \eqref{ypetit}, \eqref{e113}, nous avons
\begin{equation}
\label{e115}
\Upsilon(x,y;a,q)= \frac{\Upsilon_q(x,y)}{\varphi(q)} \Big \{  1 + O \Big( \ee^{-c_1u/(\log u)^4}  +Y_{\varepsilon} ^{-1} \Big)    \Big \}.
\end{equation}\par 
En l'absence de caractère de Siegel modulo $q$, on peut substituer $u$ à $u/(\log u)^4$ dans le terme d'erreur de \eqref{e115}. 
\end{thm}
Le résultat suivant, dont la démonstration repose sur \eqref{Gra}, fournit une évaluation de $\Upsilon(x,y;a,q)$ pour les grandes valeurs de $y$.

\begin{thm} \label{T5}
Pour une valeur convenable de $c_2>0$, tout $\varepsilon>0$, et sous les conditions $x \geqslant y \geqslant (\log x)^{2+\varepsilon} $, $q\leqslant \sqrt{y}$, et $(a,q)=1$, nous avons
\begin{equation}\label{e116}
\Upsilon(x,y;a,q)= \frac{\Upsilon_q(x,y)}{\varphi(q)} \Big \{  1 + O \Big( \frac{\log q}{u^{c_2}\log y} + \frac{1}{\log y} \Big)    \Big \}.
\end{equation}
\end{thm}

\begin{rem} Dans le domaine de validité commun à \eqref{e115} et \eqref{e116}, le terme d'erreur de \eqref{e115} est le plus précis des deux.  La discontinuité qualitative observée d'un domaine à l'autre reflète la disparité des méthodes employées.
\end{rem}

\begin{rem}\label{r6} Ainsi qu'évoqué plus haut, les démonstrations des Théorèmes \ref{T3}, \ref{T4} et \ref{T5} peuvent être adaptées pour traiter le cas  $d:=(a,q)>1$. Supposons par exemple $\mu(d)^2=1$ et $(q/d,d)=1$. Pour tout $\varepsilon >0$ et sous les conditions \eqref{ypetit} et \eqref{e113}, nous avons alors
$$  \Upsilon(x,y;a,q)= \frac{h_d(\beta)\Upsilon_{q/d}\big(x/d,y \big)}{\varphi \big(q/d \big)}  \Big \{  1+ O \Big ( \ee^{-c_1u/(\log u)^4}  +Y_{\varepsilon} ^{-1}  \Big)  \Big \}      $$
où $$ h_d(s):=  \prod_{p\, \mid  \, d}\frac{1-p^{-\nu_p s}}{1-p^{-(\nu_p+1)s}}   \quad ({\Re}e \ s>0).    $$
\end{rem}

\section{Cols} \label{Pc}

Pour $v>1$, désignons par $\xi(v)$ l'unique solution de l'équation $\ee^{\xi(v)}=1+v \xi(v)$, et  convenons que $\xi(1):=0$. Nous avons 
\begin{equation}\label{xi}
\xi(v) = \log (v\log v) +O\Big(\frac{\log_2v}{\log v}\Big) \qquad  (v\geqslant 3),
\end{equation}
et renvoyons par exemple aux articles \cite{HT93, HT93a} ou à \cite[ch. III.5]{ten} pour une description plus précise du comportement asymptotique de $\xi(v)$ lorsque $v\to\infty$. \par Posons encore  
\begin{equation}\label{e21}
  L_{\varepsilon}(y):= \ee^{(\log y)^{3/5-\varepsilon}}  \qquad( \varepsilon >0,\   y \geqslant 2  ).  
  \end{equation} 
 Il est établi dans \cite{HT} que, pour tout $\varepsilon >0$, nous avons
\begin{equation}\label{alpha}
\alpha=1-\frac{\xi(u)}{\log y} + O \Big(   \frac{1}{u(\log y)^2}+\frac{1}{L_{\varepsilon}(y)}  \Big) \qquad \big( (\log x)^{1+\varepsilon}<y\leqslant x  \big).
\end{equation} 

Les estimations suivantes du col $\beta$ ont été obtenues dans \cite{T15}.

\begin{lem} Soit $\varepsilon >0$. Nous avons
\begin{equation}\label{beta}
\beta(x,y)=  \begin{cases} \displaystyle
1-\frac{\xi(u)}{\log y}+O \Big( \frac{1}{u(\log y)^2}+\frac{1}{L_{\varepsilon}(y)} \Big) & \big( (\log x)^{1+\varepsilon}<y\leqslant x  \big) ,\\ \noalign{\vskip1mm}
\displaystyle \frac{1+O \big( 1/\log y \big)}{\log y} \log  \big(1+\eta \big)  &  \big( 2 \log x<  \psi( y) \ll (\log x)^3 \big) .
\end{cases}  
\end{equation}
\end{lem}
\par

Compte tenu de \eqref{xi}, nous pouvons noter, à fins de référence ultérieure, que 
\begin{equation}\label{1-beta}
     1-\beta \asymp  \frac{\log ( 2 u)}{\log y} \qquad (\psi(y) >2 \log x  )  .         
\end{equation}

\begin{lem}
Uniformément pour $x\geqslant y\geqslant 2$, $\psi(y)>2 \log x$, nous avons
\begin{align}
&\frac{y^{1-\beta}-1}{1-\beta} \asymp \log x,\label{y(1-beta)} \\
&y^{1-\beta}\asymp u \log 2u \label{y}.
\end{align}
\end{lem}

\begin{proof}
Lorsque $y \geqslant (\log x)^2$, l'estimation \eqref{y(1-beta)} résulte du lemme 3 de \cite{HT} et de la première formule  \eqref{beta}. Sous la condition $2 \log x < \psi(y) \ll(\log x)^3$, elle découle de l'estimation (2.16) de \cite{T15} et de \eqref{1-beta} sous la forme $(1-\beta)\asymp 1$. La relation \eqref{y} est une conséquence immédiate. 
\end{proof}

\section{Lemmes}

\subsection{Majoration du terme d'erreur $D_q$}

\begin{lem}\label{Eq}
Soit $\varepsilon>0$. Sous la condition \eqref{e161}, nous avons 
\begin{equation}\label{mEq}
D_q \ll \frac{\sqrt{u}}{(\log 2u)^{1/2+\varepsilon}}\cdot
\end{equation}
\end{lem}

\begin{proof}
Lorsque $\eta\leqslant 1$, l'estimation $D_q\asymp \omega(q)$ résulte de la Remarque~\ref{remsDqEq}(a). Comme cette hypothèse implique $y\ll \log x$ et donc $\log y \asymp \log 2u$, il suit
$$   D_q\asymp \omega(q) \ll (u\log y)^{1/2-(1+\varepsilon)(\log_2 2u)/\log 2u} \ll \frac{\sqrt{u}}{(\log 2u)^{1/2+\varepsilon}} .   $$\par 
Lorsque $\eta>1$, la condition \eqref{e161} équivaut à  $$\vartheta_q\leqslant \frac{1}{2}-(1+\varepsilon)\frac{\log_2 (2u)}{\log 2u}+O \Big( \frac{1}{\log y} \Big) , $$ 
d'où \eqref{mEq}, en reportant dans \eqref{e15}.
\end{proof}

\subsection{Estimations relatives à la fonction $g_q(s)$}

Posons 
$$ \gamma_q(s):= \log g_q(s)=\sum_{p\, | \, q}\big\{  \log \big(1-p^{-s} \big)  -\log\big(1-p^{-(\nu_p+1)s}\big)\big\} \qquad (\Re e\, s>0),$$
où, dans le membre de droite, les logarithmes complexes sont pris en détermination principale. Introduisons également la quantité 
\begin{equation}
\label{defCq}
C_q:= \min \{  \omega(q),\Delta_q^2 \}
\end{equation}
et, dans toute la suite, convenons que $u_0$ désigne une constante absolue assez grande.

\begin{lem}\label{l32}\
 Sous les conditions \eqref{e161} et
\begin{equation}\label{e32a}
\psi(y)>2\log x,  \quad u\geqslant u_0,
\end{equation}
nous avons
\begin{equation}\label{e32}
\gamma_q'(\beta) \ll D_q \log y, \qquad  -\gamma_q''(\beta) \ll C_q(\log y)^2.
\end{equation} 
Sous l'hypothèse supplémentaire $\eta \ll 1$, nous avons 
\begin{equation} \label{e33}
 \begin{aligned} 
\displaystyle \gamma_q'(\beta) &= \tfrac{1}{2} \omega(q) \log y  +O \big( \eta \omega(q) \log y + \log q \big)    ,    \\
\displaystyle \gamma_q''(\beta) &= \tfrac{1}{12} \omega(q) (\log y)^2 + O \big(  \eta \omega(q) (\log y)^2  +(\log q) \log y  \big).
 \end{aligned}
 \end{equation}
\end{lem}

\begin{proof} 
Ainsi qu'il a été observé au lemme 3.13 de \cite{RT}, les fonctions $\gamma'_q$ et $-\gamma''_q$ sont décroissantes et positives sur $]0,\infty[$. Nous avons
\begin{equation*}
\begin{aligned}
\gamma'_q(\sigma) &= \sum_{p \, \mid \, q} \Big \{  \frac{\log p}{p^{\sigma}-1} - \frac{(\nu_p+1)\log p}{p^{(\nu_p+1)\sigma}-1} \Big \}, \\
-\gamma''_q(\sigma) &= \sum_{p \, \mid \, q} \Big \{  \frac{(\log p)^2 p^{\sigma}}{(p^{\sigma}-1)^2} - \frac{(\nu_p+1)^2 (\log p)^2 p^{(\nu_p+1)\sigma}}{(p^{(\nu_p+1)\sigma}-1)^2} \Big \}.
\end{aligned} 
\end{equation*}
Un développement limité à l'ordre 3 en 0 fournit donc
\begin{equation*}
\begin{aligned}
\gamma'_q(0)&=\tfrac12\sum_{p\, | \, q}\nu_p\log p\ll \omega(q)\log y,    \\
-\gamma''_q(0)&=\tfrac1{12}\sum_{p \, |  \, q}\nu_p(\nu_p+2)(\log p)^2\ll\omega(q)(\log y)^2 .
\end{aligned}
\end{equation*}
Pour achever la preuve de \eqref{e32}, il reste  à montrer que 
\begin{equation}
\label{e35}
\gamma_q^{(j)}(\beta)\ll (\Delta_q\log y)^j \qquad (j=1,2).
\end{equation}
\par 
À cette fin, nous pouvons supposer que $\eta>1$, puisque, dans le cas contraire, nous avons $\Delta_q\gg 1+\omega(q)$, ainsi qu'observé à la Remarque  \ref{remsDqEq}(a). Nous utiliserons les majorations  
\begin{equation}
\label{majgam12}
\gamma_q'(\sigma) \ll \frac{z_q^{1-\sigma}-1}{(1-z_q^{-\sigma})(1-\sigma)}, \quad -\gamma_q''(\sigma) \ll \frac{(z_q^{1-\sigma}-1)\log z_q}{(1-z_q^{-\sigma})^2(1-\sigma)}  \qquad (0<\sigma<1), 
\end{equation}
établies au lemme 3.13 de \cite{RT}.
\par 
Considérons d'abord le cas $y \leqslant (\log x)^2$. Parallèlement à la preuve du  lemme 3.15 de \cite{RT}, nous observons que les relations \eqref{majgam12}, \eqref{beta}, \eqref{1-beta} et \eqref{y} impliquent 
\begin{equation*}
\begin{aligned}
\gamma_q'(\beta)  &\ll  y^{(1-\beta)\vartheta_q} \Big(1+\frac{1}{\beta\log z_q}\Big) \ll (\log x)^{\vartheta_q}\Big(1+\frac{1}{\vartheta_q\log (1+\eta)}  \Big)=\Delta_q\log y,\\
-\gamma_q''(\beta) &\ll (\log x )^{\vartheta_q}\Big( 1+ \frac{1}{ \beta\log z_q}  \Big)^2 \vartheta_q  \log y=(\Delta_q \log y)^2 \frac{\vartheta_q \log y}{(\log x)^{\vartheta_q}}\ll (\Delta_q \log y)^2,
\end{aligned}
\end{equation*}
ce qui fournit bien \eqref{e35}.

Lorsque $y > (\log x)^2$,  compte tenu de la Remarque \ref{remsDqEq}(b), l'estimation \eqref{e35} est une conséquence directe de \cite[lemme~3.14]{RT} sous la forme
$$\gamma_q^{(j)}(\beta)\ll \Delta_q\vartheta_q^{j-1}(\log y)^j\quad(j=1,2)$$ et de la majoration $\vartheta_q \ll \Delta_q$, énoncée à la formule (2.10) du même article.\end{proof}

\subsection{Estimations impliquant les séries $Z_q(s,y)$}
Posons
\begin{equation}\label{phij}
\varphi_{j,q}(s,y):=(-1)^{j} \frac{\dd^{j-1}}{\dd \sigma^{j-1}}\frac{Z_q'(s,y)}{Z_q(s,y)},\qquad \sigma_{j,q}:= \varphi_{j,q}(\beta,y) \qquad (j \geqslant 1, q  \geqslant 1),
\end{equation}
et convenons d'omettre le second indice lorsque $q=1$.
 Le lemme suivant est une adaptation relative aux quantités \eqref{phij}  de \cite[lemme 2.5]{T15}, correspondant au cas $q=1$.
\begin{lem}\label{l33}
Soit $j \in \mathbb{N}^*$. Sous les conditions \eqref{e161} et \eqref{e32a}, nous avons
\begin{equation}\label{e37}
\sigma_{j,q} \ll u(\log y)^j.
\end{equation}
Lorsque $j=3$, on peut multiplier le membre de droite par $\min (1,\beta \log y)+1/\sqrt{u}$. De plus, dans les mêmes hypothèses,
\begin{equation}
\label{sig2q}
\sigma_{2,q}=\sigma_2\Big\{1+O\Big(\frac1{\log u}\Big)\Big\}.
\end{equation}
Sous les conditions supplémentaires $q=1$ et $\psi(y)\ll (\log x)^3$, nous avons 
\begin{equation}\label{e38}
\sigma_2= \Big \{ 1+ O \Big(  \frac{1}{\log y} \Big)   \Big\}\frac{1+\eta}{2+\eta}u(\log y)^2. 
\end{equation}
\end{lem}

\begin{proof}
Lorsque $ y \ll (\log x)^3$, toutes les majorations découlent directement  de \cite[lemme 2.5]{T15} et des inégalités $\sigma_{j,q}\leqslant \sigma_j$ $(j \geqslant 1)$. Dans le domaine $y >(\log x)^2 $, la même approche fonctionne, {\it mutatis mutandis}, en utilisant \eqref{y}. L'égalité \eqref{e38} coïncide avec  \cite[formule (2.21)]{T15}. \par
 Il reste à établir \eqref{sig2q}. Par \eqref{e32}, nous avons
\begin{equation*}
\sigma_2-\sigma_{2,q}=-\gamma_q''(\beta)\ll  C_q(\log y)^2. 
\end{equation*}
Comme la majoration \eqref{mEq} implique $\omega(q) \ll  \sqrt{u} $ lorsque $\psi(y) \ll \log x$ et $$\Delta_q^2\ll u/\log u$$ dans la circonstance complémentaire, nous obtenons bien l'évaluation annoncée.
\end{proof}

Rappelons la définition \eqref{e112} de $Y_{\varepsilon}$ $(\varepsilon>0)$.
\par 
La démonstration du lemme suivant étend à $Z_q(s,y)$ celle de \cite[lemme 2.6]{T15}.
\par \smallskip
\begin{lem}\label{l34}
Il existe une constante  $c_3>0$ telle que, pour tout $\varepsilon>0$, et sous  les conditions \eqref{e161}, \eqref{e32a}, nous ayons
\begin{equation}\label{e39}
  \bigg |   \frac{Z_q(\beta + i \tau,y)}{Z_q(\beta,y)}  \bigg | \leqslant   \begin{cases} \displaystyle
\ee^{-c_3u(\tau \log y)^4} & (|\tau|\leqslant {1}/{\log y}) ,\\ \noalign{\vskip1mm}
\displaystyle \ee^{-c_3u  \tau^4/ (1+\tau^4)} & ( {1}/{\log y} < |\tau|  \leqslant Y_{\varepsilon} ).
\end{cases}
\end{equation}
\end{lem}
\begin{proof}
Posons $s=\beta+i\tau$ $(\tau \in \mathbb{R})$ et notons $\Vert z \Vert $ la distance du nombre réel $z$ à l'ensemble des entiers. Les calculs des pages 342-343 de \cite{T15} fournissent, pour une constante convenable $\kappa>0$,
$$  \left | \frac{Z_q(s,y)}{Z_q(\beta,y)} \right | \leqslant  \ee^{- \kappa W_q} $$
avec
$$   W_q:= \sum_{\substack{p \leqslant y \\  p \, \nmid \,  q }}\frac{\Vert {(\tau/ 2 \pi)\log p}  \Vert ^4 }{p^{\beta}}\cdot   $$
Lorsque $|\tau | \leqslant 1/\log y $, et puisque \eqref{e161} implique $z_q \leqslant y^{3/4}$, nous avons, grâce à~\eqref{y}, 
\begin{equation}
\begin{aligned}\label{W1}
\sum_{ p \, \mid \,  q }\frac{\Vert {(\tau/ 2 \pi)\log p}  \Vert ^4 }{p^{\beta}} &\ll \tau^4 (\log z_q)^3\frac{z_q^{1-\beta}-1}{1-\beta} \\&\ll  \tau^4 (\log y)^3  \frac{y^{3(1-\beta)/4}-1}{1-\beta} \ll   \frac{ (\tau \log y)^4 u^{3/4}}{(\log 2u)^{1/4}}  \cdot 
\end{aligned}
\end{equation}
Comme la minoration (2.27) de \cite{T15} implique  
\begin{equation}
\label{W2}
W_1  \gg   \frac{\tau^4(\log y)^3( y^{1-\beta}-1 )}{1-\beta}\gg u (\tau \log y)^4 ,
\end{equation}
et comme $u$ peut être choisi assez grand en vertu de \eqref{e32a}, il suit  $W_q\gg u(\tau \log y)^4,$ ce qui implique bien \eqref{e39} dans ce cas.
\par
Lorsque $1/\log y \leqslant |\tau | \leqslant \sqrt{y} $, l'argument employé dans la démonstration de \cite[lemme~5.12]{OT13} fournit
\begin{equation}\label{e310}
\sum_{\substack{z/2<p \leqslant z \\ p \, \nmid \,  q   }} \Vert   (\tau /2 \pi )\log p  \Vert^4 \gg \frac{\tau^4}{1+\tau^4}\sum_{\substack{z/2 < p \leqslant z  \\  p \, \nmid  \, q }}1   \qquad  ( y^{3/4} < z \leqslant y   )    .
\end{equation}
Or, l'hypothèse \eqref{e114} implique $\omega(q) =o\big(z / \log z  \big)$. Il suit  
\begin{equation}\label{e311}
   \sum_{\substack{z/2 < p \leqslant z  \\  p \, \nmid  \,  q }}1  \gg  \frac{z}{\log z},      
\end{equation}
d'où $W_q \gg \tau^4 u/(1+\tau^4) $ par sommation d'Abel, compte tenu de \eqref{y(1-beta)} --- cf. \cite{T15}.\par

Lorsque $\sqrt{y}< |\tau| \leqslant Y_{\varepsilon} $, le terme de gauche de \eqref{e310} est 
$$  \gg  \sum_{\substack{z/2 < p \leqslant z  \\  p \, \nmid \,  q }} \sin^4  \big( \tfrac12 \tau \log p  \big)=  \tfrac38 \sum_{\substack{z/2 < p \leqslant z  \\  p \, \nmid  \, q }} 1  -\tfrac12 \sum_{\substack{z/2 < p \leqslant z  \\  p\,  \nmid \, q }}  \big\{ \cos(\tau \log p)-\tfrac14 \cos  (2\tau \log p)  \big \}.      $$\par 
 Comme $\omega(q)=o\big(z/\log z  \big)$, l'argument de \cite[lemme 2.6]{T15} (relatif au cas $q=1 $) permet de montrer que  la dernière somme est $o\big(z/\log z  \big)$.
 Compte tenu de \eqref{e311}, une sommation d'Abel permet de conclure.
\end{proof}

\begin{lem}\label{l35}
 Il existe une constante $c_4>0$ telle que, pour tout $\varepsilon>0$ et sous les conditions \eqref{e161}, \eqref{e32a} et $1 \leqslant z \leqslant \min  \{ Y_{\varepsilon},\ee^{ c_4 u }  \}$, nous ayons
\begin{equation}\label{e312}
\Upsilon_q(x+x/z,y)-\Upsilon_q(x,y) \ll  x^{\beta} Z_q(\beta,y)/z.
\end{equation}
\end{lem}

\begin{proof}
La preuve est identique à celle de \cite[lemme $2.7$]{T15} en remplaçant $Z$ par $Z_q$, et en faisant appel au Lemme \ref{l34}. 
\end{proof}

\subsection{Estimations relatives aux séries $Z(s,\chi;y)$}

Pour tout caractère de Dirichlet $\chi$ non principal modulo $q$,  posons
\begin{equation*}
Z(s,\chi;y) := \prod_{p \leqslant y} \frac{1-\chi(p)^{\nu_p+1}p^{-(\nu_p+1)s}}{1-\chi(p)p^{-s}}  \qquad  ({\Re}e(s) >0),
\end{equation*}
et, sous la condition supplémentaire $|\tau | \leqslant Y_{\varepsilon} $,
définissons également
\begin{equation}\label{e313}
 \rW_q  (y,\tau;\chi):= \sum_{\substack{p \leqslant y \\ p \, \nmid  \, q }} \frac{\{1- {\Re}e (\chi(p)/p^{i \tau})\}^2}{p^{\beta}},
\end{equation}
et  convenons d'omettre l'indice $q$ lorsque $q=1$.\par 
Nous nous proposons ici d'établir les propositions auxiliaires nécessaires aux démonstrations des Théorèmes \ref{T3} et \ref{T4}.  \par

\par 

La preuve de l'énoncé suivant est analogue à celle de \cite[lemme 2.6]{T15}.

\begin{lem}\label{l36}
Pour un choix convenable des constantes $c_0>0$, $c_5>0$,  nous avons
\begin{equation}
\label{majZs}
\bigg |    \frac{Z(\beta +i \tau,\chi;y)}{Z_q(\beta,y)}     \bigg | \leqslant  \ee^{-c_5 \rW_q(y,\tau;\chi)}  \qquad  \Big( \tau\in\R, \ q \leqslant y^{c_0/\log_2 y}  \Big).
\end{equation}

\end{lem}

\begin{proof}
Soit $s:= \beta +i\tau$.  Un calcul de routine fournit
\begin{equation} \label{e314}
\bigg |    \frac{Z(s,\chi;y)}{Z_q(\beta,y)}     \bigg |^2 = \prod_{\substack{p \leqslant y \\ p \, \nmid \, q}}\frac{1+4\sin^2 \big( \tfrac12 \alpha_p(\nu_p+1)\big) / \{  p^{\beta(\nu_p+1)}(1-p^{-\beta(\nu_p+1)})^2 \} }{1+4\sin^2(\tfrac12 \alpha_p)/\{ p^{\beta}(1-p^{-\beta})^2 \}}
\end{equation}
où $\alpha_p\in ]-\pi,\pi]$ est l'argument du nombre complexe $\chi(p)/ p^{i\tau}$.
 \par 
Posons $\tau_p:=\Vert \alpha_p/2\pi \Vert$ et $B_p:= p^{\beta}(1-p^{-\beta})^2$. En raisonnant comme dans la preuve  de \cite[lemme 2.6]{T15}, on obtient que le terme général du produit de \eqref{e314} n'excède pas
$$   1-\frac{8 \tau_p^2 \sin^2( \pi \tau_p)}{3B_p+ 12 \sin^2(\pi \tau_p)} \cdot  $$
L'inégalité $ \vert \sin \pi \tau_p \vert   \geqslant 2 \tau_p $ fournit donc, pour une constante convenable $\kappa_0>0$,
\begin{equation} \label{e315}
   \bigg |    \frac{Z(s,\chi;y)}{Z_q(\beta,y)}     \bigg |^2 \leqslant \ee^{-\kappa_0 \rV_q}   
   \end{equation}
où l'on a posé
$$  \rV_q:= \sum_{\substack{p\leqslant y \\ p \, \nmid  \, q }}\frac{\tau_p^4}{B_p+4 \tau_p^2}   \gg    \sum_{\substack{p\leqslant y \\ p \, \nmid  \, q }}\frac{\sin^4(\pi \tau_p)}{p^{\beta}}=\sum_{\substack{p\leqslant y \\ p \, \nmid  \, q }}\frac{(1-\cos (\alpha_p))^2}{4p^{\beta}} .  $$
Cette minoration implique bien le résultat annoncé. \end{proof} 

Posons à présent 
\begin{equation}\label{e317}
\rD(y,\tau;\chi):=  \sum_{p \leqslant y} \frac{\big\{1-{\Re}e(\chi(p)/p^{i\tau}) \big\} \log p}{p^{\beta}}  , \quad  S(y,\tau;\chi) := \sum_{n \leqslant y} \frac{\chi(n) \Lambda(n)}{n^{\beta+i\tau}}\cdot 
\end{equation}\par 
\begin{lem}\label{l37}
Sous la condition \eqref{e11}, nous avons 
\begin{equation}\label{e318}
\rD(y,\tau;\chi)=\frac{y^{1-\beta}}{1-\beta}-\Re e\big(S(y,\tau;\chi)\big)+O\bigg( \frac{y^{1-\beta}}{1-\beta} \bigg \{  \ee^{-\sqrt{\log y}} + \frac 1u   \bigg \}    \bigg).
\end{equation}
\end{lem}

\begin{proof}
Compte tenu de  \eqref{1-beta} et \eqref{y}, une forme forte du théorème des nombres premiers fournit par sommation d'Abel --- \cf\ \cite[Lemme 3.5]{RT} --- 
\begin{equation*}
\sum_{p\leqslant y}\frac{\log p}{p^{\beta}}=\frac{y^{1-\beta}}{1-\beta}\bigg \{ 1+ O  \bigg( \ee^{-\sqrt{\log y}} + \frac {1}{u\log 2u}  \bigg)   \bigg \} .
\end{equation*}
De plus, 
\begin{align*} 
\sum_{p \leqslant y}  \frac{\chi(p)\log p}{p^{\beta+i \tau}} -S(y,\tau;\chi) &\ll \sum_{\substack{p^{\nu} \leqslant y \\ \nu \geqslant 2 }}  \frac{\log p}{p^{\nu \beta}} \ll \big(1+y^{1/2-\beta}\big)\log y\\ & \ll\frac{y^{1-\beta}}{1-\beta}\Big(\frac1u+\frac{\log u}{\sqrt{y}\log y}\Big) .  
\end{align*}
La formule annoncée découle de ces estimations.
\end{proof}
 \par 

Le lemme suivant est consacré à l'étude de $S(y,\tau;\chi)$ par l'approche développée dans \cite[lemme III.5.16 ]{ten}. Nous notons $\chi_1$ l'éventuel caractère de Siegel de module $q>1$ et $\beta_1$ le zéro correspondant.  
Posons $$\eta(T):=1/\{(\log T)^{2/3}(\log_2T)^{1/3}\},\quad \eta_q(T):=\min(1/\log q,\eta(T)),$$ de sorte que le théorème de Vinogradov-Korobov (cf. par exemple, \cite[ch. 9, p.~176]{Mo94}),  garantit que, pour une constante positive convenable $c_6>0$, le domaine $$\{\sigma+i\tau:|\tau|\leqslant T,\sigma>1-c_6\eta_q(T)\}$$ est une région sans zéro de $L(s,\chi)$ lorsque $\chi\neq\chi_1$. Pour $c_7$ assez grande, $c_0>0$ assez petite, et $q\leqslant y^{c_0/\log_2y}$, il existe donc, dès que $y$ est assez grand, au plus un zéro de $\prod_{\chi\neq\chi_0}L(s,\chi)$ dans la demi-bande horizontale $\sigma>1-3c_7(\log_2y)/\log y$, $|\tau|\leqslant Y_\varepsilon$.  Nous posons
\begin{equation}
\label{e319}
\vartheta(\chi):=\begin{cases}
 1 & \mbox{ si } \chi=\chi_1 \mbox{ et } \beta_1>1-c_7(\log_2y)/\log y\\
0 & \mbox{ dans le cas contraire.}
\end{cases}
\end{equation}

\begin{lem}\label{l38}
 Soit $\varepsilon >0$ assez petit. Lorsque $q$ satisfait \eqref{e113} et sous la condition 
\begin{equation}
\label{e320} 
 \chi \neq \chi_0, \quad |\tau| \leqslant Y_{\varepsilon}, \quad x\geqslant 2,\quad  \psi(y)>2 \log x,\quad y\leqslant \ee^{(\log x)^\varepsilon}, \end{equation} 
nous avons
\begin{equation}
\label{321}
S(y,\tau; \chi)=-\vartheta(\chi) \frac{y^{\beta_1-\beta-i\tau}}{\beta_1-\beta-i\tau} +O \Big(  \frac{y^{1-\beta}}{(\log y)^3}  \Big).
\end{equation}
\end{lem}
\begin{proof}
 Nous pouvons supposer $x$, et donc $y$, assez grand. Notons $$w:=\beta+i\tau, \quad v := 1-\beta+1/\log y.$$ Il résulte de la formule de Perron effective (cf. \cite[cor. II.2.4]{ten}) que, sous la condition $| \tau | \leqslant Y_\varepsilon$, nous avons
\begin{equation}\label{e322}
S(y,\tau;\chi)  = \frac{-1}{2i\pi} \int_{v -iY_\varepsilon}^{v+i Y_\varepsilon}\frac{L'(s+w,\chi)y^s}{L(s+w,\chi)s}\mathrm{d}s + O \Big(  \frac{y^{1-\beta}(\log y)^3}{Y_\varepsilon}  \Big).
\end{equation}\par 
Pour $\varepsilon$ assez petit, l'hypothèse \eqref{e320}  implique
 $$1-\beta>3c_7(\log_2y)/\log y,$$ en vertu de \eqref{1-beta}. \par 
Déplaçons alors le segment d'intégration vers la gauche jusque $$\delta:=1-\beta-c_7(\log_2y)/\log y>2 c_7(\log_2y)/\log y.$$ Lorsque $\chi\neq\chi_1$, la zone traversée ne contient aucun pôle de l'intégrande, alors qu'elle en contient exactement un,  à savoir $s=\beta_1-w$ lorsque $\chi =\chi_1$. De plus, dans les deux cas, le segment translaté est  à distance $\geqslant c_7(\log_2y)/\log y $ d'un pôle de l'intégrande.\par 
 Le théorème des résidus implique alors que l'intégrale de \eqref{e322} vaut 
\begin{equation}\label{e323}
\frac{-1}{2i\pi} \int_{\gotW}\frac{L'}{L}(s+w,\chi)\frac{y^s}{s}\mathrm{d}s - \vartheta(\chi)\frac{y^{\beta_1-w}}{\beta_1-w}  
\end{equation}
où $\gotW$ désigne la ligne brisée $[v-iY_\varepsilon,\delta-iY_\varepsilon,\delta+iY_\varepsilon,v+iY_\varepsilon]$ parcourue dans le sens trigonométrique inverse.
\par 
Compte tenu du placement de $\gotW$ par rapport aux pôles de l'intégrande, on peut établir classiquement, via la formule explicite en fonction des zéros  (cf., par exemple, \cite[pp. 380-381]{ten}, que
\begin{equation*}
\left |  \frac{L'}{L}(s+w,\chi)    \right | \ll (\log qY_\varepsilon)^2\qquad (s\in\gotW).
\end{equation*}
\par 
L'intégrale de \eqref{e323} est donc
\begin{equation*}
\label{majinttrans}
\begin{aligned}
\ll(\log yY_\varepsilon)^3y^\delta+\frac{y^{1-\beta}(\log y)^3}{Y_\varepsilon}\ll\frac{y^{1-\beta}}{(\log y)^3},
\end{aligned}
\end{equation*}
quitte à augmenter la valeur de $c_7$.
Cela complète la démonstration.
\end{proof}

\begin{lem}
\label{l39}
Sous les conditions \eqref{e113} et \eqref{e320}, nous avons
\begin{equation}
\label{e326}
\rD(y,\tau;\chi) \gg \frac{y^{1-\beta}}{(1-\beta) \{1+\vartheta(\chi)(\log_2x)^2 \} } \cdot 
\end{equation}
\end{lem}

\begin{proof}
Conservons les notations de la démonstration du Lemme \ref{l38}.
Par \eqref{e318} et \eqref{321}, nous pouvons écrire, sous les hypothèses effectuées,
\begin{equation*}
\rD(y,\tau;\chi) =\frac{y^{1-\beta}}{1-\beta}+\vartheta(\chi) {\Re}e \Big( \frac{y^{\beta_1-\beta-i\tau}}{\beta_1-\beta-i\tau}  \Big)  +O\bigg( \frac{y^{1-\beta}}{1-\beta} \bigg \{ \frac{1}{(\log y)^3} + \frac1{u}  \bigg \}    \bigg).
\end{equation*}
Cela implique immédiatement \eqref{e326} si $\vartheta(\chi)=0$. Dans le cas contraire, nous avons $\chi=\chi_1$, $\beta_1 > 1- c_7 \log _2 y/ \log y$ et $\beta_1-\beta\gg (\log_2x)/\log y $.
\par 
Si $|\tau | \leqslant 1/\log y$, posant $\varphi:=\tau\log y-\arctan\{\tau/(\beta_1-\beta)\}$ il suit
$$    {\Re}e \Big( \frac{y^{\beta_1-\beta-i\tau}}{\beta_1-\beta-i\tau}  \Big) =\frac{y^{\beta_1-\beta}\cos\varphi}{|\beta_1-\beta+i\tau|}\geqslant 0,    $$
d'où \eqref{e326}.\par 
Si $1/\log y < |\tau |\leqslant Y_{\varepsilon} $, nous avons, pour une constante convenable $c_3>0$, 
\begin{equation}
\label{Re}
\begin{aligned}
\Big |   {\Re}e  \Big(  \frac{y^{\beta_1-\beta-i\tau}}{\beta_1-\beta-i\tau}  \Big)  \Big |  &\leqslant   \Big |   \frac{y^{\beta_1-\beta}}{\beta_1-\beta-i\tau}  \Big |\leqslant \frac{y^{1-\beta}}{|1-\beta+i/\log y|}\\&\leqslant \frac{y^{1-\beta}}{1-\beta}\Big\{1-\frac{c_3}{(\log_2x)^2}\Big\}.
\end{aligned}
\end{equation}
\vskip-7mm\end{proof}

\section{Preuve du Théorème \ref{T2} } 
Plaçons-nous dans les hypothèses \eqref{e110} avec, disons, $\varepsilon \in ]0, \tfrac13[ $,  et  posons 
\begin{equation}\label{e51}
f_q(s):= \prod_{p \, \mid \, q} \big( 1-p^{-s}  \big)  \qquad (s\in \CC).
\end{equation}\par 

Lorsque $x$ est suffisamment grand, nous avons $\alpha > \tfrac12+\tfrac15\varepsilon$ d'après \eqref{alpha}. Il suit
\begin{equation}\label{e52}
0\leqslant \Psi_q(x,y)- \Upsilon_q(x,y)  \leqslant   \sum_{\substack{p \leqslant y \\ p\,  \nmid \, q }} \Psi_q \Big( \frac{x}{p^{\nu_p +1}},y \Big).
\end{equation}\par 

Sous la condition supplémentaire $y \leqslant x^{1/3}$, les théorèmes 2.4.(i) et 2.1 de \cite{RT} fournissent
\begin{equation}\label{e53}
 \Psi_q \Big( \frac{x}{p^{\nu_p +1}},y \Big)  \ll  \frac{f_q(\alpha)}{p^{(\nu_p+1)\alpha}}\Psi(x,y)\ll \frac{\Psi_q(x,y)}{p^{(\nu_p+1)\alpha}}\cdot 
\end{equation}\par 

En vertu des inégalités \eqref{e52} et \eqref{e53}, il suit
\begin{equation}\label{e54}
\Psi_q(x,y)- \Upsilon_q(x,y)  \leqslant  \Psi_q(x,y)S,
\end{equation} 
où l'on a posé
 \begin{equation}
 \label{e55}
S:=\sum_{\substack{p \leqslant y \\ p \, \nmid \, q }}\frac{1}{p^{(\nu_p+1)\alpha}}  \ll  \frac{y^{1/2-\alpha
}}{\log y} \ll \frac{u \log 2u}{\sqrt{y}\log y},
\end{equation}
d'après \cite[formule (31)]{T15}. Cela implique bien \eqref{e111} dans le cas considéré.  \par
\goodbreak
Lorsque $x^{1/3} \leqslant y \leqslant x$, nous avons $u\leqslant 3$ et $\Psi_q(x,y)\ll f_q(\alpha)\Psi(x,y)\asymp f_q(\alpha)x$ en vertu de \cite[th. 2.4(i)]{RT} et, par exemple, \cite{D}. De plus, on déduit aisément de \eqref{alpha} que $f_q(\alpha)\asymp f_q(1)=\varphi(q)/q$. La majoration triviale $\Psi_q \big(x/p^{\nu_p+1},y \big)\leqslant x/p^{\nu_p+1}$ fournit donc
\begin{equation*}\label{e57}
\begin{aligned}
\Psi_q(x,y) - \Upsilon_q(x,y) &\leqslant \sum_{p \leqslant y} \frac{x}{p^{\nu_p+1}}\ll \frac{\Psi(x,y)}{\sqrt{y}\log y}\ll\frac{\Psi_q(x,y)}{f_q(\alpha)\sqrt{y}\log y}\ll\frac{q\Psi_q(x,y)}{\varphi(q)\sqrt{y}\log y}\cdot
\end{aligned} 
\end{equation*}
Cela complète la preuve de \eqref{e111}.

\section{Preuve du Théorème \ref{T1}}

Le lemme suivant étend à $\Upsilon_q(x,y)$ la formule (3.4) de \cite{T15}. \par 

\begin{lem}\label{l40} Soit $c_4$ la constante apparaissant au Lemme \ref{l35}. Pour un choix convenable de $c_8>0$, et tout $\varepsilon>0$, nous avons, sous les hypothèses \eqref{e16} et \eqref{e161},
\begin{equation} \label{e41}
\Upsilon_q(x,y)= \frac{1}{2i\pi} \int_{\beta-iT}^{\beta+iT}\frac{Z_q(s,y) x^s}{s}\mathrm{d}s + O \Big( \frac{x^{\beta}Z_q(\beta,y)\log T}{T} \Big)
\end{equation}
avec $T:= \ee^{2c_8( \log u)^{4/3} } \leqslant \min (Y_{1/20},\ee^{c_4 u} )$.  
\end{lem}

\begin{proof}
Une version effective de la formule de Perron (cf. \cite[thm. II.2.3]{ten}) permet de montrer  que la différence entre le membre de gauche de \eqref{e41} et le terme principal du membre de droite est
\begin{equation*}
\begin{aligned}
&\ll \frac{1}{T} \sum_{\substack{n \geqslant 1\\ n\in U(y)}}  \frac{x^\beta}{n^{\beta}\{1+T|\log (x/n)|\}}\\&\ll \frac{x^\beta Z_q(\beta,y)}{T}+\sum_{ |h|\leqslant T}\frac{x^\beta}{1+|h|}\Big\{\Upsilon_q\big(x\ee^{(h+1)/T},y\big)-\Upsilon_q\big(x\ee^{h/T},y\big)\Big\}\\
&\ll \frac{x^\beta Z_q(\beta,y)}{T}\Big\{1+\sum_{|h|\leqslant T}\frac1{|h|+1}\Big\},
\end{aligned}
\end{equation*}
d'après le Lemme \ref{l35}. Cela implique bien \eqref{e41}.
\end{proof}

Pour $1\leqslant U\leqslant V$, désignons par $I_q(U,V)$ la contribution  du domaine $U<|\tau|\leqslant V$ à l'intégrale de \eqref{e41}. Nous posons $T_0:= (\log u)^{(1+\varepsilon)/2}/\{\sqrt{u}\log y\}$.

\begin{lem}\label{l41}  Pour un choix convenable de $c_9>0$ et tout $\varepsilon>0$, nous avons, sous les hypothèses \eqref{e16} et \eqref{e161},
\begin{equation}\label{e43}
I_q(T_0,T) \ll x^{\beta} Z_q(\beta,y)  \ee^{-c_9 (\log 2u)^{1+\varepsilon}}.
\end{equation}
\end{lem}
\begin{proof}
Posons $T_1 := u^{-1/3}/\log y$, $T_2:=u^{-1/5}/\log y$. D'après le 
 Lemme \ref{l34}, nous pouvons écrire
$$   I_q(T_2,T)  \ll  x^{\beta}Z_q(\beta,y)\big \{I_{21}+I_{22}\big\},   $$
avec
$$I_{21}:=\int_{T_2}^{1/\log y}\ee^{-c_3u(\tau \log y)^4}\frac{\dd\tau}{\beta+\tau}  \ll   \ee^{-c_3u(T_2 \log y)^4} \log u \ll\ee^{-{\scriptscriptstyle\frac12} c_3u^{1/5}}         $$
et
$$ \begin{aligned}
 I_{22}&:=\int_{1/\log y}^T \ee^{-c_3u \tau^4/ (1+\tau^4) }\frac{\dd\tau}{\tau}  \ll    x^{-{\scriptscriptstyle\frac12}c_3u/(\log y)^4}\log_2y +  \ee^{-{\scriptscriptstyle\frac12}c_3u}(\log u)^{4/3} .  
 \end{aligned} $$
Il s'ensuit que
\begin{equation}\label{e44}
 I_q(T_2,T)  \ll  x^{\beta}Z_q(\beta,y) \ee^{-c_9(\log (2u))^{1+\varepsilon}} .
\end{equation}\par

 Lorsque $T_1< |\tau| \leqslant T_2 $, nous avons $\tau^2 \sigma_{2,q}\gg u^{1/3}$ et $\tau^4 \sigma_{4,q}\ll u^{1/5}$, en vertu  Lemme \ref{l33}. Il résulte alors d'un développement limité que
\begin{equation}\label{e45}
\begin{aligned}
Z_q(s,y)x^s &=Z_q(\beta,y)x^{\beta}\ee^{i\tau {\gamma'_q}(\beta)  -\tau^2\sigma_{2,q}/2+i \tau^3\sigma_{3,q}/6+ \sigma_{4,q}\tau^4/24+O(1)}  \\
&  \ll   Z_q(\beta,y)x^{\beta} \ee^{-c_{10}u^{1/3}},
\end{aligned}
\end{equation}
ce qui implique
\begin{equation}\label{e46}
I_q(T_1,T_2) \ll x^{\beta} Z_q(\beta,y)\ee^{-c_{11}u^{1/3}}.
\end{equation}\par 

Lorsque $T_0 <  |\tau | \leqslant T_1 $, nous avons $\tau^2 \sigma_{2,q}\gg (\log (2u))^{1+\varepsilon} $ et $\tau^4 \sigma_{4,q}\ll u^{-1/3}$ d'après le Lemme \ref{l33}. Un développement limité fournit donc
\begin{equation}\label{e461}
x^s Z_q(s,y) \ll x^{\beta} Z_q(\beta, y) \ee^{   -{\tau^2 \sigma_{2,q}}/{2}+ {\tau^4  \sigma_{4,q}}/{4}+O(1)   } 
 \ll x^{\beta}Z_q(\beta,y)  \ee^{-c_{12} (\log (2u))^{1+\varepsilon}}.
\end{equation}
L'estimation \eqref{e43} découle de \eqref{e44}, \eqref{e46} et \eqref{e461}.
\end{proof}

Nous sommes à présent en mesure de compléter la preuve du Théorème \ref{T1}.
D'après \eqref{e41} et \eqref{e43}, nous avons
\begin{equation}  \label{e47}
\Upsilon_q(x,y) = I_q(0,T_0) + O \Big( x^{\beta}Z_q(\beta,y)\ee^{-c_9(\log (2u))^{1+\varepsilon}} \Big). 
\end{equation}
\par 
Évaluons à présent le terme principal de \eqref{e47}.\par
Puisque $T_0^j\sigma_{j,q} \ll 1$ lorsque $j=3,4$ et $\varphi_{4,q}(\beta+i \tau,y) \ll \sigma_4^*:= u(\log y)^4$ pour $|\tau|\leqslant 1/\log y$ et $T_0 \gamma'_q(\beta) \ll 1$ grâce à \eqref{mEq} et \eqref{e32}, nous pouvons écrire
\begin{equation}
\label{I}
\begin{aligned}
I_q(0,T_0)& =\frac{x^{\beta}Z_q(\beta,y)}{2\pi}\int_{-T_0}^{T_0} \ee^{i\tau {\gamma_q}'(\beta)-\tau^2 \sigma_{2,q}/2 +i \tau^3 \sigma_{3,q}/6+O(\tau^4 \sigma_4^*)  }\frac{\mathrm{d}\tau}{\beta+i\tau}\\
 &= \frac{x^{\beta}Z_q(\beta,y)}{2\pi}\Big \{ J_0+iJ_1 +\tfrac{1}{6}iJ_2 +O(K)    \Big \}
\end{aligned}
\end{equation}
où
\begin{equation*}
\begin{aligned}
J_0 & := \int_{-T_0}^{T_0}\ee^{-\tau^2\sigma_{2,q} /2} \frac{\mathrm{d}\tau}{\beta + i \tau}, \quad  J_1  := \gamma'_q(\beta) \int_{-T_0}^{T_0} \tau \ee^{-\tau^2 \sigma_{2,q}/2 }\frac{\mathrm{d}\tau}{\beta+i \tau}, \\
 J_2 & :=\sigma_{3,q} \int_{-T_0}^{T_0} \tau^3 \ee^{-\tau^2 \sigma_{2,q}/2 }   \frac{\mathrm{d}\tau}{\beta+i \tau}, \\
 K     &:= \int_{-T_0}^{T_0} \ee^{-\tau^2 \sigma_{2,q}/2} \Big( \tau^6 \sigma_{3,q}^2 + \tau^4 \sigma_4^*  + \tau^2 \gamma'_q(\beta)^2   \Big)  \frac{\mathrm{d}\tau}{ \beta +|\tau|} \cdot 
\end{aligned}
\end{equation*}\par 

Des calculs semblables à ceux de la fin de la proposition 2.13 de \cite{RT02}, utilisant le Lemme~\ref{l33}, fournissent
\begin{equation}
\label{J0}
  J_0= 2\pi G \big(\beta \sqrt{\sigma_{2,q}}\big)+O(\ee^{-c_9u^{1/3}}), \qquad  
J_2 \ll \frac{1}{u(1+\beta \sqrt{\sigma_{2,q}})}. 
\end{equation}\par 

Par le changement de variables $v=\tau\sqrt{\sigma_{2,q}}$, nous pouvons écrire
\begin{equation}\label{K}
\begin{aligned}
K&= \int_{-T_0 \sqrt{\sigma_{2,q}}}^{T_0 \sqrt{\sigma_{2,q}}}\ee^{-v^2/2}\Big(  \frac{\sigma_{3,q}^2v^6}{\sigma_{2,q}^3}  + \frac{\sigma_4^*v^4}{\sigma_{2,q}^2} + \frac{\gamma_q'(\beta)^2v^2}{\sigma_{2,q}}   \Big)\frac{\dd v}{\beta \sqrt{\sigma_{2,q}}+|v|}\\
 & \ll \frac{1+ D_q^2}{u(1+\beta\sqrt{\sigma_{2,q}})},
 \end{aligned}
\end{equation}
la dernière majoration découlant des Lemmes \ref{l32} et \ref{l33}.
\par\goodbreak
En tenant à nouveau compte du fait que l'intégrale d'une fonction impaire sur un intervalle symétrique par rapport à l'origine  est nulle, le même changement de variables fournit
\begin{equation*}
\begin{aligned}
J_1 &=  \frac{-i\gamma_q'(\beta)}{\sqrt{\sigma_{2,q}}}\int_{-T_0 \sqrt{\sigma_{2,q}}}^{T_0\sqrt{\sigma_{2,q}}}\frac{v^2 \ee^{-v^2/2}\d v}{\beta^2\sigma_{2,q} +v^2},  \\
 &=  \frac{-i\gamma_q'(\beta)}{\sqrt{\sigma_{2,q}}} \Bigg \{  \int_{-T_0 \sqrt{\sigma_{2,q}}}^{T_0 \sqrt{\sigma_{2,q}}} \ee^{-v^2/2}\d v  -\beta^2 \sigma_{2,q}  \int_{-T_0\sqrt{\sigma_{2,q}}}^{T_0 \sqrt{\sigma_{2,q}}}\frac{\ee^{-v^2/2}\d v}{\beta^2\sigma_{2,q} +v^2}  \Bigg \}\cdot
\end{aligned}
\end{equation*}
 Nous avons 
\begin{equation*}
\begin{aligned}
 &\frac1{\sqrt{2\pi}}\int_{-T_0 \sqrt{\sigma_{2,q}}}^{T_0 \sqrt{\sigma_{2,q}}} \ee^{-v^2/2}\d v = 1+O\Big(\ee^{-(\log u)^{1+\varepsilon/2}}\Big) ,\\
&\beta\sqrt{\sigma_{2,q}}\int_{-T_0\sqrt{\sigma_{2,q}}}^{T_0 \sqrt{\sigma_{2,q}}}\frac{\ee^{-v^2/2}\d v}{\beta^2\sigma_{2,q} +v^2}=J_0,
\end{aligned}
\end{equation*}
et donc 
\begin{equation}
\label{iJ1}
\frac{iJ_1}{2\pi}=\frac{\gamma_q'(\beta)}{\sqrt{2\pi\sigma_{2,q}}}\Big\{1-\beta\sqrt{2\pi\sigma_{2,q}}G \big(\beta \sqrt{\sigma_{2,q}}\big)\Big\}+O\Big(\ee^{-(\log u)^{1+\varepsilon/2}}\Big).
\end{equation}\par 
En reportant \eqref{J0}, \eqref{K} et \eqref{iJ1} dans \eqref{I} puis \eqref{e47}, nous obtenons
\begin{equation}
\label{e411}
\Upsilon_q(x,y)
= x^{\beta}Z_q(\beta,y)  G(\beta  \sqrt{\sigma_{2,q}}) \Big \{     1+ R_q(\beta) +O  \Big(   \frac{1+D_q^2}{u}   \Big)   \Big \}
\end{equation}
avec 
$$   R_q(\beta):=\beta  \gamma_q'(\beta) \Big(   \frac{1}{\sqrt{2 \pi \sigma_{2,q}}\beta G(\beta \sqrt{\sigma_{2,q}}) }-1 \Big)     .  $$ \par 

Comme \eqref{e14} implique
\begin{equation}\label{e412}
    \frac{1}{\sqrt{2\pi} zG(z) }-1 \asymp \frac{1}{z(z+1)}  \qquad (z>0)   , 
    \end{equation}
il vient, par \eqref{e32} et \eqref{beta},
\begin{equation}\label{e413}
 R_q(\beta) \ll    \frac{D_q}{\sqrt{u}(1+\beta\sqrt{\sigma_{2,q}} )}\ll\frac{D_q(1+\eta)}{\sqrt{u}+\eta u} \cdot 
\end{equation}\par 

Pour obtenir \eqref{e17}, il suffit donc de montrer que l'on peut remplacer $\sigma_{2,q}$ par $\sigma_2$ dans \eqref{e411} et \eqref{e413} sans altérer le terme d'erreur. \par 

 Au vu de \eqref{sig2q}, c'est clair dans le cas de \eqref{e413}. Pour traiter le cas de \eqref{e411}, observons que l'on a $\sigma_{2,q}-\sigma_2\ll C_q(\log y)^2$ en vertu de la seconde majoration \eqref{e32}, et donc $$\sqrt{\sigma_{2,q}}-\sqrt{\sigma_2}\ll  C_q(\log y)^2/\sqrt{\sigma_2}\ll C_q(\log y)/\sqrt{u}.$$
Comme il résulte de \eqref{e14} que  
\begin{equation} \label{e414}
\frac{ G'(z)}{G(z)}=z-\frac{1}{\sqrt{2\pi}G(z)} \asymp \frac{-1}{1+z}  \qquad (z>0),
\end{equation} 
nous obtenons, par la formule des accroissements finis,
 \begin{equation}\label{e415}
G(\beta \sqrt{\sigma_{2,q}})=\Big\{1+O\Big(\frac{ C_q \beta \log y}{\sqrt{u}(1+\beta \sqrt{\sigma_2} )}\Big)\Big\}G\big(\beta\sqrt{\sigma_2}\big).
\end{equation}\par 
Par \eqref{sig2q} et l'évaluation de $C_q=\min(\omega(q),\Delta_q^2)$ résultant de la Remarque \ref{remsDqEq}(a), le terme d'erreur de \eqref{e415} est en toute circonstance 
$$\ll \frac{C_q\beta\log y}{\sqrt{u}+u\beta\log y}\ll \frac{C_q\eta}{\sqrt{u}+\eta u}\ll\begin{cases}\displaystyle\frac{D_q\eta}{\sqrt{u}+\eta u}& \mbox{ si } \eta\leqslant 1\\
\noalign{\vskip-3mm}\\
\displaystyle\frac{D_q^2}u& \mbox{ si } \eta>1.
\end{cases}$$ 
Cela établit bien  \eqref{e17}.
\par 
Pour prouver \eqref{e18},
précisons le terme d'erreur de \eqref{e17} lorsque $\eta \leqslant \tfrac12$.  Grâce à \eqref{thm1}, nous pouvons  supposer $q \geqslant 2$.  Nous pouvons également supposer $\eta$ arbitrairement petit. D'après \eqref{beta}, nous avons  $\beta \log y \leqslant 1$ dès que $y$ est assez grand. \par  
En vertu de \eqref{e33}, nous pouvons écrire
\begin{equation}\label{e416}
    \gamma_q'(\beta) =  \tfrac12 \omega (q) \log y \big \{   1 + O \big(\mathfrak{R} \big)   \big \}    
\end{equation}    
avec $\mathfrak{R}:= \eta + (\log q)/\{\omega(q) \log y \} $. Nous pouvons donc remplacer le signe $\ll$ par $\asymp$ dans l'estimation \eqref{e413}.\par 

De plus, la seconde estimation \eqref{e33} permet  d'écrire 
\begin{equation}\label{e417}
\begin{aligned}
\sigma_2- \sigma_{2,q} 
&= \tfrac{1}{12}\omega(q)(\log y)^2 \big  \{1 +O (\mathfrak{R} ) \big \}.
\end{aligned}
\end{equation}\par 

Il suit
\begin{equation}\label{e418}
\sqrt{\sigma_{2}}- \sqrt{\sigma_{2,q}} \asymp \frac{\omega(q) \log y}{\sqrt{u}},
\end{equation}
d'où, par \eqref{e414}, 
\begin{equation}\label{e419}
   \frac{G(\beta \sqrt{\sigma_{2,q}})}{G(\beta \sqrt{\sigma_2} )}-1 \asymp \frac{\eta \omega(q)}{\sqrt{u}(1+\eta\sqrt{u} )} \cdot
\end{equation}   
La formule \eqref{e18} découle de ces observations, en reportant dans \eqref{e411}.\par 

Prouvons à présent  l'estimation \eqref{e19}, relative au cas $\eta=o(1/\sqrt{u})$. Il résulte de \eqref{e418} que 
\begin{equation*}
\sqrt{\sigma_{2,q}}=\sqrt{\sigma_2}  \Big\{ 1+O \Big( \frac{\omega(q)}{u} \Big)   \Big \}
\end{equation*}
et de \eqref{e38} que 
\begin{equation*}
\sqrt{\sigma_2} =\frac{\sqrt{u}\log y}{\sqrt{2}} \Big \{  1+O \Big(  \eta +\frac{1}{\log y}  \Big) \Big \}.
\end{equation*}
Grâce à \eqref{e416}, il suit
\begin{equation*}
\frac{\gamma_q'(\beta)  }{\sqrt{2\pi \sigma_{2,q}} }= \frac{\omega(q)}{2\sqrt{\pi u}} \big \{  1+ O \big(  \mathfrak{R}  \big)   \big \}.
\end{equation*}\par 

Compte tenu de \eqref{e14}, \eqref{e411} et \eqref{e419}, nous obtenons bien \eqref{e19} lorsque $$\beta\sqrt{\sigma_2}\asymp \eta\sqrt{u}=o(1).$$ 
\par

\section{Preuve du Théorème \ref{T3}}
Plaçons-nous dans les hypothèses \eqref{ypetit} et \eqref{e113}.\par 
Parallèlement à \eqref{e41}, nous pouvons écrire, pour $\chi \neq \chi_0 $ et $T := Y_{2 \varepsilon}  $,
\begin{equation}\label{e70}
\Upsilon(x,y;\chi) = \frac{1}{2i\pi}  \int_{\beta -iT}^{\beta+i T} {Z}(s,\chi;y) x^s \frac{\mathrm{d}s}{s} + O \Big( \frac{x^{\beta}Z_q(\beta,y) \log T}{{T}}  \Big)  ,\end{equation}
et, d'après \eqref{majZs}, pour une constante convenable $\kappa >0$, 
\begin{equation}\label{e71}
|Z(s,\chi;y)|  \leqslant Z_q(\beta,y) \ee^{-\kappa \rW_q(y,\tau;\chi)}\qquad (s=\beta+i\tau,\, |\tau | \leqslant T).
\end{equation}
où $\rW_q(y,\tau;\chi)$ a été défini en \eqref{e313}. \par 
Grâce à \eqref{y(1-beta)}, une sommation d'Abel --- \cf\ \cite[lemme 3.6]{RT} --- fournit également, sous la seule condition \eqref{e11}, 
\begin{equation}
\label{sompb}
\sum_{p\leqslant y}\frac1{p^\beta}\asymp u.
\end{equation}
\par 
La quantité $\rD(y,\tau;\chi)$ étant définie en \eqref{e317}, nous déduisons donc de \eqref{1-beta}, \eqref{y}, et \eqref{e326}, via l'inégalité de Cauchy-Schwarz, que 
\begin{equation}
\label{W}
  \rW(y,\tau; \chi)   \geqslant  \dfrac{    \rD(y,\tau;\chi)^2 }{   (\log y)^2  \sum_{p\leqslant y}{p^{-\beta}}}\gg    \frac{u}{  1 +\vartheta(\chi)(\log u)^4  }, 
   \end{equation}
où $\vartheta(\chi)$ a été défini par \eqref{e319}.
\par 

Rappelons la notation \eqref{zq} et observons  que l'hypothèse \eqref{e113} implique \mbox{$z_q\ll \log y$}, et donc, compte tenu de \eqref{1-beta}, pour tout $\varepsilon>0$,
\begin{equation}\label{pmidq}
\sum_{p \, \mid \, q}\frac{1}{p^{\beta}}\ll  \frac{z_q^{1-\beta}-1}{(1 - \beta)\log z_q}\ll_\varepsilon {u^{\varepsilon}}.     
\end{equation}
\par 

Compte tenu de \eqref{W}, il vient, pour un choix convenable de $c_2>0$,
\begin{equation}\label{e72} 
|  Z(s,\chi;y )  | \leqslant Z_q(\beta,y) \ee^{ -c_2u(\chi)}     \qquad  (|\tau| \leqslant Y_{  2 \varepsilon}),
\end{equation}
où l'on a posé $u(\chi):=u/\{ 1+\vartheta(\chi)(\log u)^4\}$.
\par 
En reportant  \eqref{e72} dans \eqref{e70}, nous obtenons
\begin{eqnarray*}
\Upsilon(x,y;\chi) & \ll & x^{\beta} Z_q(\beta,y) \Big \{  (\log y)^{3/2} \ee^{ -c_2u(\chi)}    + Y_{\varepsilon}^{-1}   \Big \} ,
\end{eqnarray*}
et donc, compte tenu de \eqref{UP-ygrand},
\begin{eqnarray*}
   \Upsilon(x,y;\chi) &\ll & \Upsilon_q(x,y) \sqrt{u}\log y  \Big \{  (\log y)^{3/2} \ee^{ -c_2u(\chi)}  + Y_{\varepsilon}^{-1}   \Big \}.
 \end{eqnarray*}
La condition supplémentaire \eqref{ypetit} permet  de conclure. 

\section{Preuve du Théorème \ref{T4}}

La relation d'orthogonalité des caractères permet d'écrire
$$    \Upsilon(x,y;a,q)=\frac{1}{\varphi(q)}  \Big \{  \Upsilon_q(x,y) + \sum_{\chi\neq \chi_0} \overline{\chi(a)} \Upsilon(x,y;\chi)   \Big \}   .   $$
L'estimation annoncée découle donc immédiatement du Théorème \ref{T3}.

\section{Preuve du Théorème \ref{T5} }

Plaçons-nous dans les hypothèses de l'énoncé.
Il résulte de \eqref{e111} et de l'estimation $q/\varphi(q)\ll \log _2 q$ que
\begin{equation}\label{e62}
\Psi_q(x,y)-\Upsilon_q(x,y) \ll \frac{ q u \log (2u) \Psi_q(x,y)}{\varphi(q)\sqrt{y}\log y }\ll \frac{\Psi_q(x,y)}{(\log y)^2}  \cdot
\end{equation}
De plus
\begin{equation}\label{e63}
\begin{aligned}
\Psi(x,y;a,q) -\Upsilon(x,y;a,q) &  \leqslant \  \sum_{\substack{p \leqslant y \\  p \, \nmid \, q }}  \Psi \Big( \frac{x}{p^{\nu_p+1}},y;a/p^{\nu_p+1},q  \Big)  \\
         & \leqslant \sum_{p \leqslant y} \sup_{(b,q)=1} \Psi \Big( \frac{x}{p^{\nu_p+1}},y;b,q  \Big)
\end{aligned}
\end{equation}
où, dans la première inégalité, $a/p^{\nu_p+1}$ désigne la classe des entiers $b$ tels que $$p^{\nu_p+1}b\mmd aq.$$ \par
D'après \eqref{Gra}, \eqref{e53}, sous la condition supplémentaire $y\leqslant x^{1/3}$, chaque terme de la dernière somme est 
$$\ll\Psi_q\Big( \frac{x}{p^{\nu_p+1}},y  \Big)\ll\frac{f_q(\alpha)\Psi(x,y)}{p^{(\nu_p+1)\alpha}}\ll\frac{\Psi_q(x,y)}{p^{(\nu_p+1)\alpha}}\ll\frac{\Upsilon_q(x,y)}{p^{(\nu_p+1)\alpha}}\cdot$$
Les estimations \eqref{e55} et  \eqref{Gra} impliquent alors  \eqref{e116}.\footnote{Il est à noter que dans ce cas, la condition sur $q$ peut être relâchée en $q\leqslant y^{1-\varepsilon}$.} 
\par  

Lorsque $x^{1/3}< y \leqslant x$, nous déduisons simplement de \eqref{e63} que
$$    \Psi(x,y;a,q) - \Upsilon(x,y;a,q) \ll \sum_{ p\leqslant y}\frac{x}{p^{\nu_p+1}}\ll \frac{\Psi(x,y)}{ \sqrt{y}\log y}\ll\frac{\Psi_q(x,y)}{\varphi(q)\log y}  \cdot  $$
Nous obtenons encore \eqref{e116}.\par 

\bibliographystyle{abbrvfr}
{\footnotesize
\bibliography{mabiblio.bib}}
\bigskip
\leftskip10mm{\parindent0mm \footnotesize Institut Élie Cartan,\par 
 Université de Lorraine \\ BP 70239 \\  54506 Vand\oe{}uvre-lès-Nancy Cedex, \par 
France\par \smallskip
\noindent \tt cecile.dartyge\at univ-lorraine.fr, \par 
david.feutrie\at univ-lorraine.fr, \par 
gerald.tenenbaum\at univ-lorraine.fr\par }
\end{document}